\documentclass[11pt]{amsart}
\usepackage[T1]{fontenc}
\usepackage{amsmath,amssymb,amsthm}
\usepackage{mathtools}
\usepackage{hyperref}
\usepackage[nameinlink]{cleveref}
\usepackage{autonum}
\usepackage{tikz}
\usetikzlibrary{positioning}
\usepackage{booktabs}
\usepackage{longtable}
\usepackage{float}
\usepackage{listings}
\usepackage{xcolor}
\usepackage{caption}
\usepackage{fvextra}
\usepackage{fancyvrb}

\DefineVerbatimEnvironment{Verbatim}{Verbatim}{
  breaklines=true,
  breakanywhere=true
}
 \allowdisplaybreaks
\newtheorem{thm}{Theorem}[section]

\newtheorem{prop}{Proposition}[section]
\newtheorem{cor}{Corollary}[section]
\newtheorem{eg}{Example}[section]
\newtheorem{definition}{Definition}[section]

\newtheorem{rem}{Remark}[section]

\newtheorem{prob}{Problem}[section]

\crefname{thm}{Theorem}{Theorem}
\crefname{lem}{Lemma}{Lemma}
\crefname{prop}{Proposition}{Proposition}
\crefname{cor}{Corollary}{Corollary}
\crefname{eg}{Example}{Example}
\crefname{definition}{Definition}{Definition}
\crefname{prob}{Problem}{Problem}

\def\fr#1{\mathfrak{#1}}
\def\bb#1{\mathbb{#1}}

\def\ali#1{\begin{align}#1\end{align}}

\def\w{\wedge}
\def\bw{\textstyle{\bigwedge}}
\def\o{\oplus}
\def\bo{\bigoplus}
\def\ang#1{\langle #1\rangle}
\def\dim#1{{\rm dim\ }#1}
\def\0{\{ 0\}}
\def\Ker#1{{\rm Ker}\ #1}
\def\Im#1{{\rm Im}\ #1}

\title[mixed Hodge structures on $\fr{n}_k(\bb{R})$]{mixed Hodge structures on the strictly upper triangular Lie algebras}
\author{Taito Shimoji}
\date{}
\address{Department of Mathematics, Graduate School of Science, The University of Osaka, Osaka, Japan}
\email{u215629i@ecs.osaka-u.ac.jp}
\begin{document}
\maketitle
\begin{abstract}
We show that, for $k\geq 3$, a lattice in the unitriangular Lie group
$U_k(\bb{R})$ is the fundamental group of a smooth quasi-projective
complex algebraic variety if and only if $k=3$.
The proof uses the conjugation symmetry and weight conditions of mixed Hodge structure on the Malcev Lie algebra proven by Morgan.
This result identifies a family of groups that must be excluded
in the search for smooth quasi-projective varieties with torsion-free
nilpotent fundamental groups of high nilpotency class.
\end{abstract}
\section{Introduction}
In this paper, we consider characterization problem for lattices in Lie groups to be the fundamental groups of smooth quasi-projective varieties. In particular, we focus on the case such that the Lie grups are simply-connected and nilpotent. Here, a \emph{lattice} in a simpy-connected Lie group means a discrete-cocompact subgroup.
\subsection{Main Result}
For $k\geq 3$, let $U_k(\bb{R})$ denote the Lie group defined by
\[
U_k(\bb{R}):=\{A=(a_{ij})\in GL_k(\bb{R})\mid a_{ii}=1,a_{ij}=0\text{ if }i>j\}.
\]
Note that $U_3(\bb R)$ is the $3$-dimensional Heisenberg Lie group $H_3(\bb R)$. The main result is as follows. The following result gives the characterization for lattices in $U_k(\bb R)$ to be the fundamental groups of smooth quasi-projective varieties.
\begin{thm}[\cref{main}]
Let $\Gamma$ be a lattice in $U_k(\bb{R})$ for $k\geq 3$. $\Gamma$ is the fundamental group of a smooth quasi-projective variety if and only if $k=3$.
\end{thm}
It is known that there exists a smooth quasi-projective variety such that the fundamnetal group is isomorphic to any lattice in $U_3(\bb R)$. Thus ,to prove the main result, it is enough to show that
$\fr{n}_k(\bb{R})$ satisfies the conditions in the following
theorem of Morgan only in the case $k=3$.
\begin{thm}\cite[Theorem {\bf (9.4)}]{Mor78}\label{Morganintro}
Let $M$ be a smooth quasi-projective variety. Let $G$ be a simply-connected real nilpotent Lie group. Assume that $\pi_1(M, x)$ is a lattice in $G$. For the Lie algebra $\fr{g}$ of $G$,  there exists a mixed Hodge structure on $(\fr{g},W,F)$ such that for the induced bigrading on the cohomology $H^*(\fr{g}_{\bb{C}})=\bo_{p, q\geq 0, p+q\geq 1}H_{p, q}^*(\fr{g}_{\bb{C}})$ of the complexification $\fr{g}_{\bb{C}}=\fr{g}\otimes\bb{C}$, we have 
\ali{
&{\bf (W)}:
\begin{array}{l}
H^1(\fr{g}_{\bb{C}})
 = H_{1,0}^1(\fr{g}_{\bb{C}}) \o H_{0,1}^1(\fr{g}_{\bb{C}}) \o H_{1,1}^1(\fr{g}_{\bb{C}}) \text{ and }\\
H^2(\fr{g}_{\bb{C}})
 = H_{2,0}^2(\fr{g}_{\bb{C}}) \o H_{1,1}^2(\fr{g}_{\bb{C}}) \o H_{0,2}^2(\fr{g}_{\bb{C}})
 \o H_{2,1}^2(\fr{g}_{\bb{C}}) \o H_{1,2}^2(\fr{g}_{\bb{C}}) \o H_{2,2}^2(\fr{g}_{\bb{C}}).
\end{array}
}
\end{thm}
We show the main theorem by proving the following proposition.
In the following proposition, we show that $k=3$ or $4$. We then exclude the case $k=4$ using the author's previous work
\cite[Proposition 3.2]{Shi25}.
\begin{prop}[\cref{general}]\label{generalintro}
Let \(k\geq3\). Suppose that $\fr{n}_k(\bb R)$ carries a mixed Hodge
structure satisfying the conditions in \cref{Morganintro}. Then we have $k=3$.
\end{prop}
\subsection{Comparison with the previous work}
The key point of the proof of main result is ``conjugation symmetry''.
After real splitting, a mixed Hodge structure
$(\fr{g},W,F)$ compatible with the Lie bracket gives
a bigrading
$
\fr{g}_{\bb{C}}
=
\bo_{p,q}\fr{g}_{p,q}
$
such that
\[
\overline{\fr{g}_{p,q}}=\fr{g}_{q,p}.
\]
In particular, we have
\[
\dim\fr{g}_{p,q}=\dim\fr{g}_{q,p}.
\]
The induced bigrading on cohomology also satisfies
\[
\overline{H^j_{p,q}}=H^j_{q,p},
\]
and hence
\[
\dim H^j_{p,q}=\dim H^j_{q,p}.
\]

Thus, \cref{Morganintro} implies that the Lie algebra
corresponding to a torsion-free nilpotent quasi-projective
group admits a bigrading satisfying the weight condition
{\bf (W)} and the dimension symmetry condition
\[
{\bf (H)}:\quad
\dim\fr{g}_{p,q}=\dim\fr{g}_{q,p},
\qquad
\dim H^j_{p,q}=\dim H^j_{q,p},
\quad j=1,2.
\]
However, the following remark shows that these conditions
do not capture the full restriction imposed by
a mixed Hodge structure.

\begin{rem}[\cref{bigradOK}]\label{rem:bigrading-MHS}
For every odd integer $k\geq5$, the complexification
$\fr{n}_k(\bb{R})\otimes\bb{C}$ admits a Lie algebra
bigrading
\[
\fr{n}_k(\bb{R})\otimes\bb{C}
=
\bo_{\substack{p+q\leq-1\\p,q\leq0}}\fr{n}_{p,q}
\]
satisfying conditions {\bf (W)} and {\bf (H)}.
See \cref{bigradOK} for the construction.

However, the proof of our main theorem shows that
$\fr{n}_k(\bb{R})$ admits no mixed Hodge structure
satisfying the conditions in \cref{Morganintro}.
\end{rem}

This remark shows that conditions {\bf (W)} and {\bf (H)}
alone do not imply our main theorem.
Conjugation symmetry gives a further restriction
on the Lie algebras of nilpotent quasi-projective groups.
Our proof uses this symmetry, rather than only
the equality of dimensions.
\subsection{Related Topics}
In this part, we introduce three topics related to the main theorem.
The first is Serre's problem for smooth quasi-projective
complex algebraic varieties.
The second is Campana's question on nilpotency class.
The third is comparison with nilpotent K\"ahler groups.
\subsubsection{Serre's problem for lattices in nilpotent Lie groups}
Serre proved that every finite group is the fundamental group
of a smooth projective complex algebraic variety
\cite[Proposition 15]{Ser58}. The following problem is known as Serre's problem.
\begin{prob}
Characterize the finitely presented groups that are the
fundamental groups of smooth quasi-projective complex
algebraic varieties.
\end{prob}

Every smooth quasi-projective complex algebraic variety
has the homotopy type of a finite CW-complex.
In particular, its fundamental group is finitely presented.
Serre's problem asks which finitely presented groups arise
in this way.

Since the class of finitely presented groups is broad,
one often fixes a class $\mathcal{C}$ and studies
Serre's problem within $\mathcal{C}$. For exmaple see \cite{DPS08,DPS09,FS14,Rog24}.
In this paper, we study Serre's problem for lattices
in simply connected nilpotent Lie groups.
Morgan's mixed Hodge theory gives necessary conditions
on their Malcev Lie algebras \cite{Mor78}.
Related results on nilpotent quotients under the assumption
that the quasi-Albanese map is proper are given in
\cite{AC25}.
The author studied torsion-free nilpotent quasi-projective
groups of rank at most seven in \cite{Shi25}.
The general problem for lattices in simply connected
nilpotent Lie groups remains open.

Our main theorem solves this problem for lattices
in the unitriangular Lie groups $U_k(\bb{R})$.
For every $k\geq 3$, it determines whether a lattice
in $U_k(\bb{R})$ is a quasi-projective group.
\subsubsection{Campana's question: two-step nilpotency problem}
Campana asked whether a torsion-free nilpotent
K\"ahler group can have nilpotency class at least three
\cite{Cam95}.
The same question for quasi-K\"ahler groups was also raised
by Aguilar Aguilar and Campana \cite{AC25}.
In terms of lattices, we consider the following problem.

\begin{prob}
Let $\Gamma$ be a lattice in a simply connected
nilpotent Lie group $G$.
Suppose that $\Gamma$ is the fundamental group of
a smooth quasi-projective complex algebraic variety
or a quasi-K\"ahler manifold.
Is the nilpotency class of $\Gamma$ at most two?
\end{prob}

All known examples of torsion-free nilpotent fundamental
groups in these settings have nilpotency class at most two
\cite{AC25,Rog26}.

This paper considers the following question:
if a fundamental group of higher nilpotency class exists,
what restrictions must its group structure satisfy?
We focus on lattices in the unitriangular Lie groups
$U_k(\bb{R})$.
Their nilpotency class is $k-1$, so this family contains
groups of arbitrarily high nilpotency class.
Our main theorem excludes every lattice in $U_k(\bb{R})$
for $k\geq4$ from the class of quasi-projective groups.

There has also been progress under additional geometric
assumptions.
Rogov proved the two-step bound under conditions
on higher Albanese maps
\cite[Theorem B]{Rog25}.
He also proved that the fundamental group of an aspherical
normal complex algebraic variety is virtually two-step
nilpotent if it is virtually nilpotent \cite{Rog26}.

Cao, Deng, Hacon, and P\u{a}un proved that the image
of every complex linear representation of the fundamental
group of a smooth special quasi-projective variety
is virtually nilpotent of class at most two
\cite[Theorem A]{CDHP26}.
In particular, their theorem gives a positive answer
to the above problem for smooth special quasi-projective
varieties with torsion-free nilpotent fundamental group.
Related results and examples with nonabelian two-step
nilpotent fundamental group are given in \cite{CDY25c}.

However, the existence of a smooth quasi-projective variety
or a quasi-K\"ahler manifold with torsion-free nilpotent
fundamental group of class at least three remains open.
\subsubsection{Comparison with nilpotent K\"ahler groups}
Assume that a lattce $\Gamma$ in a simply-connected nilpotent Lie group $G$ is the fundamental group of a compact
K\"ahler manifold.
Then $\fr{g}$ is quadratically presented.
This follows from the formality theorem of
Deligne--Griffiths--Morgan--Sullivan \cite{DGMS75}
and the relation between $1$-minimal models and Malcev Lie algebras
\cite{Mor78} (see also \cite{CT95,CT97}).

Let $\fr{n}_k(\bb{R})$ be the Lie algebra of strictly upper triangular
real $k\times k$ matrices.
For $k\geq 3$, this Lie algebra is not quadratically presented.
Thus, no lattice in the corresponding simply connected nilpotent
Lie group is the fundamental group of a compact K\"ahler manifold.

However, the answer is different for non-compact varieties.
Let $E$ be an elliptic curve, and let $L$ be a line bundle
of nonzero degree on $E$. Remove the zero section from $L$, and write $L^\times$
for the resulting space.
Then $L^\times$ is a smooth quasi-projective variety
and a $\bb{C}^\ast$-bundle over $E$.
Its fundamental group gives a central extension
\[
1\longrightarrow \bb{Z}
\longrightarrow \pi_1(L^\times)
\longrightarrow \bb{Z}^2
\longrightarrow 1.
\]
Its real Malcev Lie algebra is isomorphic to $\fr{n}_3(\bb{R})$ (see \cite{AC25,HZ87,Mor78}). The main theorem of this paper gives a negative answer
to the following question unless $k=3$.
\begin{prob}
Can a lattice in $U_k(\bb R)$ be the fundamental group
of a smooth quasi-projective complex variety over $\bb C$?
\end{prob}
\subsection{AI Declaration}
The author used ChatGPT for language editing and other editional assistance in preparing the manuscript.
\subsection{Acknowledgement}
The author would like to thank his supervisor, Hisashi Kasuya at Nagoya University, for helpful advice on the introduction, on improving the proof of \cref{general} and discussions that led him to consider the difference between mixed Hodge structures and bigradings. The author was supported by JSPS KAKENHI Grant Number 26KJ1608.
\section{Mixed Hodge structures and Lie algebras}
We recall the basic definitions of mixed Hodge structures
and Lie algebras in this category.
See \cite{Del71} for mixed Hodge structures,
and \cite{Mor78,Hai87} for their relation to
Malcev Lie algebras.
All vector spaces and Lie algebras in this section
are finite-dimensional over $\bb{Q}$.
For a vector space $V$, we write
\[
V_{\bb{R}}=V\otimes_{\bb{Q}}\bb{R},
\qquad
V_{\bb{C}}=V\otimes_{\bb{Q}}\bb{C}.
\]

\begin{definition}
A \emph{pure Hodge structure of weight $m$}
on a rational vector space $V$ is a decomposition
\[
V_{\bb{C}}
=
\bigoplus_{p+q=m} V^{p,q}
\]
such that
\[
\overline{V^{p,q}}=V^{q,p}.
\]
Its Hodge filtration is
\[
F^aV_{\bb{C}}
=
\bigoplus_{\substack{p+q=m\\p\geq a}}V^{p,q}.
\]
\end{definition}

\begin{definition}
A \emph{mixed Hodge structure} on $V$ consists of
a finite increasing filtration $W_\bullet$ on $V$
and a finite decreasing filtration $F^\bullet$
on $V_{\bb{C}}$.
For every integer $m$, the filtration induced by $F^\bullet$
on
\[
\operatorname{Gr}_m^W V
=
W_mV/W_{m-1}V
\]
must define a pure Hodge structure of weight $m$.

We call $W_\bullet$ the \emph{weight filtration}
and $F^\bullet$ the \emph{Hodge filtration}.
The mixed Hodge structure is pure of weight $m$ if
\[
W_{m-1}V=0,
\qquad
W_mV=V.
\]
\end{definition}

A morphism of mixed Hodge structures is a rational
linear map that preserves both filtrations.
We write $\mathrm{MHS}_{\bb{Q}}$ for this category.
It is an abelian tensor category.
Every morphism is strict for both filtrations;
see \cite[Section~2.3]{Del71}.
In particular, kernels, images, and cokernels
admit natural mixed Hodge structures.

Tensor products, duals, and exterior powers also have
natural mixed Hodge structures.
For example,
\[
W_m(V\otimes U)
=
\sum_{a+b=m}W_aV\otimes W_bU,
\]
and
\[
F^p(V\otimes U)_{\bb{C}}
=
\sum_{a+b=p}F^aV_{\bb{C}}\otimes F^bU_{\bb{C}}.
\]

\begin{definition}\label{def:MHS-Lie}
A \emph{Lie algebra with a compatible mixed Hodge structure}
is a rational Lie algebra $\fr{g}$ with a mixed Hodge structure
such that the bracket
\[
[-,-]\colon
\bw^2\fr{g}\longrightarrow\fr{g}
\]
is a morphism of mixed Hodge structures.

Equivalently, $\fr{g}$ is a Lie algebra object
in $\mathrm{MHS}_{\bb{Q}}$.
In particular,
\[
[W_a\fr{g},W_b\fr{g}]
\subset W_{a+b}\fr{g},
\qquad
[F^p\fr{g}_{\bb{C}},F^q\fr{g}_{\bb{C}}]
\subset F^{p+q}\fr{g}_{\bb{C}}.
\]
\end{definition}

\begin{rem}\label{rem:Lie-cohomology-MHS}
A compatible mixed Hodge structure on $\fr{g}$
induces a mixed Hodge structure on its Lie algebra
cohomology $H^j(\fr{g})$ for every $j$.

Indeed, consider the Chevalley--Eilenberg complex
with trivial coefficients:
\[
C^\bullet(\fr{g})
=
\left(\bw^\bullet\fr{g}^*,d\right).
\]
Each space $\bw^j\fr{g}^*$ has the mixed Hodge
structure induced by duals and exterior powers.
On degree-one elements, the differential is the negative
dual of the bracket:
\[
d\alpha(u,v)=-\alpha([u,v]).
\]
Thus,
\[
d\colon
\fr{g}^*\longrightarrow\bw^2\fr{g}^*
\]
is a morphism of mixed Hodge structures.
The differential on higher degrees is obtained
by the Leibniz rule, so it is also a morphism
of mixed Hodge structures.

Therefore,
\[
H^j(\fr{g})
=
\frac{
\Ker\bigl(d\colon C^j(\fr{g})\to C^{j+1}(\fr{g})\bigr)
}{
\Im
\bigl(d\colon C^{j-1}(\fr{g})\to C^j(\fr{g})\bigr)
}
\]
has a natural mixed Hodge structure.
This follows from the abelian category properties
recalled above.
The cup product is also a morphism of mixed Hodge structures.
\end{rem}

\begin{rem}\label{rem:real-splitting-Lie}
After applying Deligne's real splitting in \cite{Del94}, we write
\[
\fr{g}_{\bb C}=\bo_{p,q}\fr{g}_{p,q}.
\]
Since the splitting is functorial and compatible
with tensor products, the bracket is a morphism.
Thus we have
\[
[\fr{g}_{p,q},\fr{g}_{r,s}]
\subset\fr{g}_{p+r,q+s},
\qquad
\overline{\fr{g}_{p,q}}=\fr{g}_{q,p}.
\]

Give $(\fr{g}_{p,q})^*$ type $(-p,-q)$.
This defines a bigrading on the Chevalley--Eilenberg complex:
\[
C^j(\fr{g}_{\bb{C}})
=
\bigoplus_{p,q}C^j_{p,q},
\qquad
d(C^j_{p,q})\subset C^{j+1}_{p,q}.
\]
Consequently,
\[
H^j(\fr{g}_{\bb{C}})
=
\bigoplus_{p,q}H^j_{p,q},
\]
where
\[
H^j_{p,q}
=
\frac{
\Ker\bigl(d\colon C^j_{p,q}\to C^{j+1}_{p,q}\bigr)
}{
\Im \bigl(d\colon C^{j-1}_{p,q}\to C^j_{p,q}\bigr)
}.
\]
Moreover,
\[
\overline{H^j_{p,q}}=H^j_{q,p},
\qquad
W_mH^j(\fr{g}_{\bb{C}})
=
\bigoplus_{p+q\leq m}H^j_{p,q}.
\]
These components describe the real splitting
of the induced mixed Hodge structure on Lie algebra cohomology.
\end{rem}
\section{Cohomology of strictly upper triangular Lie algebras}\label{Kostant}
In this section, we recall the cohomology of strictly upper triangular Lie algebras $\fr{n}_k(\bb R)\otimes\bb C$ proven by Kostant \cite{Kos61}. We first recall Kostant's theorem for the nilradical of a Borel subalgebra. 
\begin{thm}[{\cite[Theorem 5.14, p.~362]{Kos61}}]
Let $\fr{s}$ be a complex semisimple Lie algebra,
and let
\[
\fr{b}=\fr{h}\oplus\fr{n}
\]
be a Borel subalgebra with Cartan subalgebra $\fr{h}$
and nilradical $\fr{n}$.
Let $\Delta_+$ be the corresponding set of positive roots,
let $\Delta_-=-\Delta_+$, and let $W$ be the Weyl group.

For each $\alpha\in\Delta_+$, choose a nonzero root vector
$E_\alpha$, and let $e_\alpha\in\fr{n}^*$ be its dual.
For $w\in W$, put
\[
\Phi_w=w(\Delta_-)\cap\Delta_+,
\qquad
\ell(w)=|\Phi_w|.
\]
Then the cohomology classes
\[
\left[
\bw_{\alpha\in\Phi_w}e_\alpha
\right],
\qquad
w\in W,\quad \ell(w)=j,
\]
form a basis of $H^j(\fr{n})$.
\end{thm}
We apply this theorem to $\fr{n}_k(\bb{R})\otimes\bb C$.
Let $\fr{b}_k(\bb{C})$ be the upper triangular subalgebra
of $\fr{sl}_k(\bb{C})$, and let $\fr{h}_k(\bb{C})$
be its diagonal subalgebra.
Then
\[
\fr{b}_k(\bb{C})
=
\fr{h}_k(\bb{C})\oplus (\fr{n}_k(\bb{R})\otimes\bb C),
\]
so $\fr{n}_k(\bb{R})\otimes\bb C$ is the nilradical of a Borel subalgebra.

For
\[
H=\operatorname{diag}(h_1,\dots,h_k)\in\fr{h}_k(\bb{C}),
\]
let $\varepsilon_i(H)=h_i$.
Since
\[
[H,E_{ij}]=(h_i-h_j)E_{ij},
\]
we have
\[
\Delta_+
=
\{\varepsilon_i-\varepsilon_j\mid 1\leq i<j\leq k\},
\qquad
\fr{n}_k(\bb{R})\otimes\bb C
=
\bo_{1\leq i<j\leq k}\bb{C}E_{ij}.
\]
The simple roots are given by
\[
\alpha_i=\varepsilon_i-\varepsilon_{i+1},
\qquad 1\leq i\leq k-1.
\]

Let $e_{ij}$ be the dual of $E_{ij}$, and put
\[
x_i=e_{i,i+1},
\qquad
y_i=-e_{i,i+2}.
\]
Thus $x_i$ is dual to the root vector of $\alpha_i$,
and $-y_i$ is dual to the root vector of
$\alpha_i+\alpha_{i+1}$.

The Weyl group of $\fr{sl}_k(\bb{C})$ is identified
with the symmetric group $\fr{S}_k$
\cite[Chapter I, Example 7.14]{Mil13}.
Under this identification, the simple reflection $s_i$
corresponds to the transposition $(i\;\;i+1)$.
Its elements of length two are
\[
s_i s_j\quad(j-i\geq2),
\qquad
s_i s_{i+1},
\qquad
s_{i+1}s_i.
\]
Their root sets are
\[
\begin{aligned}
\Phi_{s_i s_j}
&=\{\alpha_i,\alpha_j\},
&&j-i\geq2,\\
\Phi_{s_i s_{i+1}}
&=\{\alpha_i,\alpha_i+\alpha_{i+1}\},\\
\Phi_{s_{i+1}s_i}
&=\{\alpha_{i+1},\alpha_i+\alpha_{i+1}\}.
\end{aligned}
\]
Therefore, Kostant's theorem gives the following basis
of $H^2(\fr{n}_k(\bb{R})\otimes\bb C)$:
\[
\begin{aligned}
&[x_i\wedge x_j],
&&1\leq i<j\leq k-1,\quad j-i\geq2,\\
&[x_i\wedge y_i],\quad[x_{i+1}\wedge y_i],
&&1\leq i\leq k-2.
\end{aligned}
\]
\section{Mixed Hodge structures and the fundamental groups of smooth quasi-projective varieties}
In this section, we review necessary conditions proved by Morgan
for torsion-free nilpotent groups to be the fundamental groups
of smooth quasi-projective varieties (\cite{Mor78}).
These conditions are stated in terms of mixed Hodge structures
on Lie algebras and their cohomology.
The theorem is as follows.
\begin{thm}\cite[Theorem {\bf (9.4)}]{Mor78}\label{Morgan}
Let $M$ be a smooth quasi-projective variety. Let $G$ be a simply-connected real nilpotent Lie group. Assume that $\pi_1(M, x)$ is a lattice in $G$. For the Lie algebra $\fr{g}$ of $G$,  there exists a mixed Hodge structure on $(\fr{g},W,F)$ such that for the induced bigrading on the cohomology $H^*(\fr{g}_{\bb{C}})=\bo_{p, q\geq 0, p+q\geq 1}H_{p, q}^*(\fr{g}_{\bb{C}})$ of the complexification $\fr{g}_{\bb{C}}=\fr{g}\otimes\bb{C}$, we have 
\ali{
&{\bf (W)}:
\begin{array}{l}
H^1(\fr{g}_{\bb{C}})
 = H_{1,0}^1(\fr{g}_{\bb{C}}) \o H_{0,1}^1(\fr{g}_{\bb{C}}) \o H_{1,1}^1(\fr{g}_{\bb{C}}) \text{ and }\\
H^2(\fr{g}_{\bb{C}})
 = H_{2,0}^2(\fr{g}_{\bb{C}}) \o H_{1,1}^2(\fr{g}_{\bb{C}}) \o H_{0,2}^2(\fr{g}_{\bb{C}})
 \o H_{2,1}^2(\fr{g}_{\bb{C}}) \o H_{1,2}^2(\fr{g}_{\bb{C}}) \o H_{2,2}^2(\fr{g}_{\bb{C}}).
\end{array}
}
\end{thm}
Consequently, we have the following.
\begin{cor}\label{bigrading}
The complexification $\fr{g}_{\bb C}$ of $\fr{g}$ admits a bigrading $\fr{g}_{\bb C}=\bo_{p+q\leq -1,p,q\leq 0}\fr{g}_{p,q}$ such that this bigrading and the induced bigrading $H^*(\fr{g}_{\bb C})=\bo_{p+q\geq 0,p,q\geq 0}H_{p,q}^*(\fr{g}_{\bb C})$ satisfy the following conditions.
\ali{
\text{{\bf (W)}}:&H^1(\fr{g}_{\bb{C}})
 = H_{1,0}^1(\fr{g}_{\bb{C}}) \o H_{0,1}^1(\fr{g}_{\bb{C}}) \o H_{1,1}^1(\fr{g}_{\bb{C}}) \text{ and }\\
&H^2(\fr{g}_{\bb{C}})
 = H_{2,0}^2(\fr{g}_{\bb{C}}) \o H_{1,1}^2(\fr{g}_{\bb{C}}) \o H_{0,2}^2(\fr{g}_{\bb{C}})
 \o H_{2,1}^2(\fr{g}_{\bb{C}}) \o H_{1,2}^2(\fr{g}_{\bb{C}}) \o H_{2,2}^2(\fr{g}_{\bb{C}}).\\
\text{{\bf (H)}}:&\dim\fr{g}_{p,q}=\dim\fr{g}_{q,p},\dim H_{s,t}^j(\fr{g}_{\bb C})=\dim H_{t,s}^j(\fr{g}_{\bb C})\text{ for }j=1,2.
}
\end{cor}
\begin{eg}\label{bigradOK}
Let $\fr{n}_{2k+1}(\bb R)$ be the Lie algebra defined by
\[
\fr{n}_{2k+1}(\bb R)=\{A=(a_{ij})\mid a_{ij}=0\text{ for }i\geq j\}.
\] 
Then the matrix units $E_{ij}$, where $1\leq i<j\leq 2k+1$,
form a basis of $\fr{n}_{2k+1}(\bb{R})$ and the structure constants with respect to the basis $\{E_{ij}\}$ are all rational. Thus $E_{ij}$ forms a basis of $\fr{n}_{2k+1}(\bb R)\otimes \bb C$ as $\bb C$-Lie algebra. we define a bigrading on $\fr{n}_{2k+1}(\bb R)\otimes \bb C$. 
Assign bidegree $(-1,0)$ to $E_{2i-1,2i}$ and
bidegree $(0,-1)$ to $E_{2i,2i+1}$ for $1\leq i\leq k$. These assignments extend to a bigrading on $\fr{n}_{2k+1}(\bb{R})\otimes\bb C$ by the Lie bracket.
 Since the dual elements $e_{i,i+1}$ of 
\[
E_{i,i+1}\text{ for }1\leq i\leq 2k
\]
are in $\Ker d$, we have
\[
H^1(\fr{n}_{2k+1}(\bb R)\otimes \bb C)=\ang{[e_{i,i+1}]\mid 1\leq i\leq 2k}.
\]
By \cref{Kostant} the following cohomology
classes form a basis of $H^2(\fr{n}_{2k+1}(\bb{R})\otimes\bb C)$.
\[
\begin{aligned}
&[x_i\wedge x_j],
&&1\leq i<j\leq k-1,\quad j-i\geq 2,\\
&[x_i\wedge y_i],\quad [x_{i+1}\wedge y_i],
&&1\leq i\leq k-2.
\end{aligned}
\]
The classes $[x_i\wedge x_j]$ have bidegree
$(2,0)$, $(1,1)$, or $(0,2)$.
The classes $[x_i\wedge y_i]$ and $[x_{i+1}\wedge y_i]$
have bidegrees $(2,1)$ and $(1,2)$, respectively, when $i$ is odd,
and $(1,2)$ and $(2,1)$, respectively, when $i$ is even.
Thus, the bigrading satisfies the conditions in \cref{bigrading}.
\end{eg}
In the next section, we use the following proposition. This following gives restrictions on the weights of cohomology. For the proof, see \cite{Shi25}.
\begin{prop}\cite[Proposition 3.2]{Shi25}\label{useful}
Let $\fr{n}$ be a nilpotent Lie algebra. Assume that the complexification $\fr{n}_{\bb{C}}=\bo_{p+q\leq -1,p,q\leq 0}\fr{n}_{p,q}$ admits a bigrading satisfying the condition in \cref{Morgan}. If there exists an element $0\neq [x]\in H_{p,q}^1(\fr{n}_{\bb{C}})$ and $z_1\in \fr{n}_{-s,-t}^*$ such that $dz_1=x\w y$ for some $y\in\fr{n}_{\bb{C}}^*$, then we have
\ali{
(p+s,q+t)\in\{(2,0),(0,2),(1,1),(2,1),(1,2),(2,2)\}.
}
\end{prop}
\section{Main Result}
In this section, we prove the following theorem.
This rules out the most basic examples of nilpotent Lie algebras
of nilpotency class at least three.
\begin{thm}\label{main}
Let $\Gamma$ be a lattice in $U_k(\bb{R})$ for $k\geq 3$. $\Gamma$ is the fundamental group of a smooth quasi-projective variety if and only if $k=3$.
\end{thm}
Note that any lattice in $U_3(\bb R)$ is isomorphic to the fundamental group of a smooth quasi-projective variety (see \cite{AC25,HZ87,Mor78}). Thus it suffices to show the only if part. 
\begin{prop}\label{general}
Let \(k\geq3\). Suppose that $\fr{n}_k(\bb R)$carries a mixed Hodge
structure satisfying the conditions in \cref{Morgan}. Then we have $k=3$.
\end{prop}
\begin{proof}
Let
\[
\fr{n}_k(\bb{R})\otimes\bb{C}
=
\bo_{\substack{p+q\leq -1\\p,q\leq 0}}
\fr{n}_{p,q}
\]
be the bigrading obtained from \cref{Morgan}.
Let $e_{i,j}$ be the dual of $E_{i,j}$, and let
\[
x_i=e_{i,i+1},
\qquad
y_i=-e_{i,i+2}.
\]
Then
\[
dx_i=0,
\qquad
dy_i=x_i\wedge x_{i+1}
\quad
(1\leq i\leq k-2).
\]
The forms $x_1,\ldots,x_{k-1}$ give a basis
of the first cohomology.

By \cref{Morgan}, we can write
\[
x_i=x_i^{1,0}+x_i^{0,1}+x_i^{1,1},
\]
where
\[
x_i^{p,q}\in\fr{n}_{-p,-q}^*,
\qquad
dx_i^{p,q}=0.
\]
Since the splitting is real,
\[
\overline{x_i^{1,0}}=x_i^{0,1},
\qquad
\overline{x_i^{1,1}}=x_i^{1,1}.
\]

The homogeneous components of
$dy_i=x_i\wedge x_{i+1}$ satisfy
\ali{
dy_i^{2,0}
&=x_i^{1,0}\wedge x_{i+1}^{1,0},\\
dy_i^{0,2}
&=x_i^{0,1}\wedge x_{i+1}^{0,1},\\
dy_i^{1,1}
&=x_i^{1,0}\wedge x_{i+1}^{0,1}
  +x_i^{0,1}\wedge x_{i+1}^{1,0},\\
dy_i^{2,1}
&=x_i^{1,0}\wedge x_{i+1}^{1,1}
  +x_i^{1,1}\wedge x_{i+1}^{1,0},\\
dy_i^{1,2}
&=x_i^{0,1}\wedge x_{i+1}^{1,1}
  +x_i^{1,1}\wedge x_{i+1}^{0,1},\\
dy_i^{2,2}
&=x_i^{1,1}\wedge x_{i+1}^{1,1}.
}
Now that
\[
d(x_i^{1,0}\w y_i^{2,0})=d(x_i^{0,1}\w y_i^{0,2})=d(x_i^{1,1}\w y_i^{2,2})=0
\]
for $1\leq i\leq k-2$. By the condition {\bf (W)}, there exists $z_i^{3,0}\in\fr{n}_{-3,0}^*,z_i^{0,3}\in\fr{n}_{0,-3}^*$ and $z_i^{3,3}\in\fr{n}_{-3,-3}^*$ such that
\ali{
dz_i^{3,0}=x_i^{1,0}\w y_i^{2,0},\\
dz_i^{0,3}=x_i^{0,1}\w y_i^{0,2},\\
dz_i^{3,3}=x_i^{1,1}\w y_i^{2,2}.
}
This contradicts \cref{useful}. Thus we have
\[
y_i^{2,0}=y_i^{0,2}=y_i^{2,2}=0\text{ for }a\leq i\leq k-2.
\]
We next show that
\[
y_i^{2,1}=y_i^{1,2}=0.
\]
Suppose that $y_i^{2,1}\neq 0$ for some $1\leq i\leq k-2$. By the above equality $y_i^{2,0}=y_i^{0,2}=y_i^{2,2}=0$, we have
\ali{
d(x_i^{1,0}\wedge y_i^{2,1})
&=-x_i^{1,0}\wedge x_i^{1,1}
  \wedge x_{i+1}^{1,0}
=0,\\
d(x_{i+1}^{1,0}\wedge y_i^{2,1})
&=-x_{i+1}^{1,0}\wedge x_i^{1,0}
  \wedge x_{i+1}^{1,1}
=0.
}
By the condition {\bf (W)}, there exists $z_i^{3,1},z_{i+1}^{3,1}\in\fr{n}_{-3,-1}^*$ such that
\[
dz_i^{3,1}=x_i^{1,0}\w y_i^{2,1}\text{ and }dz_{i+1}^{3,1}=x_{i+1}^{1,0}\w y_i^{2,1}.
\] 
If $x_i^{1,0},x_{i+1}^{1,0}\neq 0$, this contradicts \cref{useful}. Thus $x_i^{1,0}=0$ or $x_{i+1}^{1,0}=0$. Then we have
\[
x_i=x_i^{1,1}\neq 0 \text{ or }x_{i+1}=x_{i+1}^{1,1}\neq 0.
\]
Since
\[
d(x_{i}^{1,1}\w y_i^{2,1})=-x_i^{1,1}\w x_i^{1,0}\w x_{i+1}^{1,1}=0
\]
and
\[
d(x_{i+1}^{1,1}\w y_i^{2,1})=-x_{i+1}^{1,1}\w x_i^{1,1}\w x_{i+1}^{1,0}=0,
\]
by the condition {\bf(W)}, there exists $z_i^{3,2},z_{i+1}^{3,2}\in\fr{n}_{-3,-2}^*$ such that
\[
dz_i^{3,2}=x_i^{1,1}\w y_i^{2,1}\text{ and }dz_{i+1}^{3,2}=x_{i+1}^{1,1}\w y_i^{2,1}.
\] 
This contradicts \cref{useful}. Thus $y_i^{2,1}=0$ for $1\leq i\leq k-2$.

We next show that
\[
x_i^{1,0}\neq 0
\qquad (1\leq i\leq k-1).
\]
Suppose that $x_i^{1,0}=0$ for some $1\leq i\leq k-2$.
Then
\[
x_i^{0,1}=0,\qquad x_i=x_i^{1,1}\neq 0.
\]
Since $y_i^{2,1}=0$, we have
\[
0=dy_i^{2,1}
=x_i^{1,1}\wedge x_{i+1}^{1,0}.
\]
Thus
\[
x_{i+1}^{1,0}=x_{i+1}^{0,1}=0,\qquad
x_{i+1}=x_{i+1}^{1,1}\neq 0.
\]
Since $y_i^{2,2}=0$, we have
\[
0=dy_i^{2,2}
=x_i^{1,1}\wedge x_{i+1}^{1,1}
=x_i\wedge x_{i+1}.
\]
This contradicts the linear independence of $x_i$ and $x_{i+1}$.
If $x_{k-1}^{1,0}=0$, the same argument applies to
$x_{k-2}$ and $x_{k-1}$.
Thus $x_i^{1,0}\neq 0$ for $1\leq i\leq k-1$.

By the equality $y_i^{2,0}=0$, we have
\[
x_i^{1,0}\wedge x_{i+1}^{1,0}=0
\qquad (1\leq i\leq k-2).
\]
Hence
\[
\dim H^1_{1,0}
=
\dim\ang{x_1^{1,0},\ldots,x_{k-1}^{1,0}}
=1.
\]
By conjugation,
\[
\dim H^1_{1,0}=\dim H^1_{0,1}=1.
\]
We next show that $\dim H^1_{1,1}\leq 1$.
Suppose that $x_i^{1,1}\neq 0$ and $x_{i+1}^{1,1}=0$.
Then
\[
0=dy_i^{2,1}
=x_i^{1,1}\wedge x_{i+1}^{1,0}\neq 0,
\]
which is a contradiction.
The same argument applies if
$x_i^{1,1}=0$ and $x_{i+1}^{1,1}\neq 0$.
Thus
$
x_i^{1,1}=0
$
 if and only if
$
x_{i+1}^{1,1}=0
$
for $1\leq i\leq k-2$. By the equality $y_i^{2,2}=0$, we have
\[
x_i^{1,1}\wedge x_{i+1}^{1,1}=0
\qquad (1\leq i\leq k-2).
\]
Therefore,
$
\dim H^1_{1,1}
=
\dim\ang{x_1^{1,1},\ldots,x_{k-1}^{1,1}}
=0$ or $1$. If $\dim H^1_{1,1}=1$, then we have
\[
k-1
=
\dim H^1(\fr{n}_k(\bb{R})\otimes\bb{C})
=3.
\]
Thus $k=4$. However, $\dim\fr{n}_4(\bb{R})=6$ and $\dim H^1(\fr{n}_4(\bb{R})\otimes\bb{C})=3$. This contradicts \cite[Proposition 3.2]{Shi25}.
Hence $\dim H^1_{1,1}=0$, and
\[
k-1
=
\dim H^1(\fr{n}_k(\bb{R})\otimes\bb{C})
=
\dim H^1_{1,0}+\dim H^1_{0,1}
=2.
\]
Therefore, $k=3$.
\end{proof}

\end{document}